\documentclass[11pt]{amsart}
\usepackage{comment}
\usepackage{amssymb,amsmath,amsthm,mathrsfs,amstext,amssymb}
\usepackage{dsfont}
\usepackage{tabularx}
\usepackage{array}
\usepackage[square,sort,comma,numbers]{natbib}
\usepackage{url} 
\usepackage{hyperref}
\usepackage{accents}
\usepackage[shortlabels]{enumitem}
\usepackage{tikz-cd}
\usepackage{adjustbox}

\newtheorem{theorem}{Theorem}
\newtheorem*{theorem*}{Theorem}

\newtheorem*{maintheorem*}{Main Theorem}
\newtheorem{lemma}[theorem]{Lemma}
\newtheorem*{lemma*}{Lemma}

\newtheorem*{fact*}{Fact}
\newtheorem{corollary}[theorem]{Corollary}
\newtheorem*{corollary*}{Corollary}

\newtheorem*{proposition*}{Proposition}
\newtheorem*{gquestion*}{Guiding Question}

\theoremstyle{definition}

\newtheorem*{definition*}{Definition}
\newtheorem{remark}[theorem]{Remark}
\newtheorem*{remark*}{Remark}
\newtheorem{question}[theorem]{Question}
\newtheorem*{question*}{Question}

\newcommand{\NN}{{\mathbb{N}}}

\newcommand{\ZZ}{{\mathbb{Z}}}

\renewcommand{\c}{{\mathfrak{c}}}

\DeclareMathOperator{\restr}{\upharpoonright}

\newcommand{\simpleset}[1]{{\{{#1}\}}}

\newcommand{\set}[2]{{\{ {#1} \mid {#2} \}}}
\newcommand{\seq}[2]{{\langle {#1} \mid {#2} \rangle}}

\newcommand{\infsubset}[1]{{[#1]^{\omega}}}

\DeclareMathOperator{\cantorspace}{{^\omega 2}}

\DeclareMathOperator{\fix}{fix}

\DeclareMathOperator{\actson}{{\curvearrowright}}

\newcommand{\bigast}{\mathop{\scalebox{1.5}{\raisebox{-0.2ex}{$\ast$}}}}

\title{Isomorphism types of definable (maximal) cofinitary groups}
\author{Lukas Schembecker}

\begin{document}	
	\maketitle
	
	\begin{abstract}
		In \cite{Kastermans_2006} Kastermans proved that consistently $\bigoplus_{\aleph_1} \ZZ_2$ has a cofinitary representation.
		We present a short proof that $\bigoplus_{\c} \ZZ_2$ always has an arithmetic cofinitary representation.
		Further, for every finite group $F$ we construct an arithmetic maximal cofinitary group of isomorphism type $(\bigast_\c \ZZ) \times F$.
		This answers an implicit question by Schrittesser and Mejak in \cite{MejakSchrittesser_2022} whether one may construct definable maximal cofinitary groups not decomposing into free products.
	\end{abstract}
	
	\section{Introduction}
	
	A cofinitary group is a subgroup $G \subseteq S_\omega$ such that every $g \in G \setminus \simpleset{e}$ only has finitely many fixpoints.
	It is maximal iff it is maximal with respect to inclusion.
	We are interested in the possible isomorphism types of (maximal) cofinitary groups.
	A full classification of all possible isomorphism types of (maximal) cofinitary groups is still open, but there are some results that realize certain groups as (maximal) cofinitary groups and conversely that maximal cofinitary groups cannot possibly have certain isomorphism types.
	Equivalently, in terms of group actions we may think about which groups may possess a (maximal) cofinitary representation when acting on $\omega$.
	The known restrictions on the possible isomorphism types of maximal cofinitary groups may be summarized as follows:
	
	\begin{theorem*}[Truss, \cite{Truss_1985}; Adeleke, \cite{Adeleke_1988}]
		Every countable cofinitary group is not maximal.
	\end{theorem*}
	
	\begin{theorem*}[Kastermans, \cite{Kastermans_2009}]
		Every cofinitary group with infinitely many orbits is not maximal.
	\end{theorem*}
	
	As a consequence of this theorem, Blass noticed the following:
	
	\begin{corollary*}[Blass, \cite{Kastermans_2006}]
		Every abelian cofinitary group is not maximal.
	\end{corollary*}

	In terms of definability, Kastermans also proved the following restriction.
	A set is $K_\sigma$ iff it is a countable union of compact sets.

	\begin{theorem*}[Kastermans, \cite{Kastermans_2006}]
		Every $K_\sigma$ cofinitary group is not maximal.
	\end{theorem*}
	
	On the positive side, Zhang's forcing \cite{Zhang_2000} may be used to force the existence of a maximal cofinitary representation of the free group in $\kappa$ generators for any $\kappa$ of uncountable cofinality.
	For uncountable $\kappa$ of countable cofinality one may use a product version of Zhang's forcing as in \cite{FischerTornquist_2015}.
	Further, as a converse to the restriction above, Kastermans \cite{Kastermans_2009} proved that consistently for every $n \in \NN$ and $m \in \NN \setminus \simpleset{0}$ one may force the existence of a maximal cofinitary group with exactly $n$ finite and $m$ infinite orbits.
	He also proved that consistently there is a locally finite maximal cofinitary group \cite{Kastermans_2009}.
	Finally, a modification of Zhang's forcing also yields:
		
	\begin{theorem*}[Kastermans, \cite{Kastermans_2006}]
		There exists a c.c.c.\ forcing which forces the existence of a cofinitary representation of $\bigoplus_{\mathfrak{\aleph_1}} \ZZ_2$.
	\end{theorem*}

	In the second section of this paper we will prove that the existence of such a cofinitary representation is not just consistent, but in fact always exists.
	Also, more generally, we prove the statement for $\bigoplus_{\mathfrak{c}} \mathbb{Z}_2$ and not just $\bigoplus_{\aleph_1} \mathbb{Z}_2$.
	We also compute its complexity as $\Sigma^0_2$.
	
	\begin{theorem*}
		There is a cofinitary representation of $\bigoplus_{\mathfrak{\c}} \ZZ_2$.
	\end{theorem*}
	
	The third section will be concerned with isomorphism types of definable maximal cofinitary groups.
	This line of research stems from the construction of a Borel maximal cofinitary group by Horowitz and Shelah \cite{HorowitzShelah_2016}.
	Schrittesser and Mejak improved their construction to obtain an arithmetic maximal cofinitary group:
	
	\begin{theorem*}[Schrittesser, Mejak, \cite{MejakSchrittesser_2022}]
		There is a $\Pi^0_1$ generating set for a free maximal cofinitary group, which is also maximal as an eventually different family of permutations.
	\end{theorem*}
	
	In their subsequent discussion, they note that all currently known constructions of definable maximal cofinitary groups yield groups decomposing into free products.
	We prove that this is not always the case due to the following theorem:

	\begin{theorem*}
		Let $G \actson \omega$ be cofinitary and maximal as an eventually different family of permutations and $F$ be a finite group.
		Then $G \times F \actson \omega \times F$ is cofinitary and maximal as an eventually different family of permutations.
	\end{theorem*}

	We may combine both theorems to obtain the following result:
	
	\begin{corollary*}
		Let $F$ be a finite group.
		Then, there is $\Pi^0_1$-generating set for a maximal cofinitary group isomorphic to $(\bigast_\c \ZZ) \times F$.
		Thus, the group itself has complexity $\Sigma^0_2$.
	\end{corollary*}
	
	The result follows by coding $\omega \times F$ in $\omega$ by elements of the form $\left|F\right|n + m$ with $m < \left|F\right|$.
	It is also easy to see how to get the generators without increasing the complexity.
	Finally, this result really yields groups not decomposing into free products: For example $(\bigast_\c \ZZ) \times \ZZ_2$ has non-trivial center, whereas any free product of non-trivial groups has trivial center.
	Also, we note that both presented proofs crucially use Cayley's theorem:
	
	\begin{theorem*}[Cayley's theorem]
		The left regular action of every group on itself is free.
	\end{theorem*}

	\section{A cofinitary representation of $\bigoplus_{\c} \ZZ_2$}

	We begin with the proof of an arithmetic cofinitary representation of $\bigoplus_{\c} \ZZ_2$.
	Remember that $\bigoplus_{\c}\ZZ_2$ has the following group representation:
	
	\begin{remark}
		The group $\bigoplus_{\mathfrak{\c}} \ZZ_2$ can be represented as the free group in $\c$-many generators modulo the relations $a^2$ and $ab = ba$ for all generators $a, b$.
		Hence, for any set $I$ we may define an action of $\bigoplus_{\mathfrak{\c}} \ZZ_2$ on $I$ by defining it on the set of generators and verifying the two types of relations above.
	\end{remark}

	\begin{theorem}
		There is a cofinitary representation of $\bigoplus_{\mathfrak{\c}} \ZZ_2$.
	\end{theorem}

	\begin{proof}
		Let $H := \bigoplus_{\mathfrak{\c}} \ZZ_2$ be generated by $\set{h_f}{f \in \cantorspace}$.
		Inductively, we will define an interval partition $\seq{I_n}{n < \omega}$ of $\omega$ and for every real $f \in 2^\omega$ an action of $h_f$ on $I_n$.
		
		Given $n < \omega$ let $H_n := \bigoplus_{s \in {^n 2}} \ZZ_2$ be generated by $\set{h_s}{s \in {^n 2}}$.
		Further, let $I_n$ be the interval above $\bigcup_{m < n} I_m$ of size $\left|H_n\right|$.
		$H_n$ acts freely on itself, so also on $I_n$.
		Thus, for further computations we may identify $I_n$ with $H_n$.
		For every $f \in \cantorspace$ let the action of $h_f$ on $H_n$ be defined as the action of $h_{f \restr n}$ on $H_n$.
		First, we show that this generates a well-defined group action of $H$ on $H_n$, for if $f \in \cantorspace$ and $h \in H_n$, then we compute
		$$
			h_f.(h_f.h) = h_{f \restr n}.(h_{f \restr n}.h) = (h_{f \restr n} \circ h_{f \restr n}).h = e.h = h.
		$$
		Now, let $f,g \in \cantorspace$ and $h \in H_n$.
		Then we compute
		$$
			h_f.(h_g.h) = h_{f \restr n}.(h_{g \restr n}.h) = (h_{f \restr n} \circ h_{g \restr n}).h = (h_{g \restr n} \circ h_{f \restr n}).h = h_{g \restr n}.(h_{f \restr n}.h) = h_g.(h_f.h).
		$$
		By the previous remark this suffices.
		Now, we define a group action of $H$ on $\omega$ for $f \in \cantorspace$ by
		$$
			h_f.k := h_{f \restr n}.k, \quad \text{ where } k \in I_n.
		$$
		We already verified that all group actions of $H$ on $I_n$ are well-defined, so we obtain a well-defined group action of $H$ on $\omega = \bigcup_{n < \omega} I_n$.
		It remains to verify that the action is cofinitary.
		So let $h \in H$.
		Choose $F_0 \subseteq \cantorspace$ finite such that $h = \sum_{f \in F_0} h_f \neq e$.
		Choose $N < \omega$ such that for all $n > N$ and $f \neq g \in F_0$ we have $f \restr n \neq g \restr n$.
		We finish the proof by showing that for $n > N$ we have that $h$ acting on $H_n$ has no fixpoints.
		
		But on one hand, by choice of $N$ we have $\sum_{f \in F_0} h_{f \restr n} \neq e$.
		On the other hand, $H_n$ acts freely on itself which implies that the action of $\sum_{f \in F_0} h_{f \restr n}$ on $H_n$ has no fixpoints.
		But $h = \sum_{f \in F_0} h_f$ acts on $H_n$ the same way $\sum_{f \in F_0} h_{f \restr N}$ does, so that also $h$ acting on $H_n$ has no fixpoints.
	\end{proof}

	\begin{remark}
		Note, that since $\bigoplus_{\mathfrak{c}} \mathbb{Z}_2$ is abelian, by Blass' Corollary above it cannot have a \underline{maximal} cofinitary representation.
		One may also directly observe that the representation in the proof above is not maximal as the action of $H$ has infinitely many orbits, namely the $I_n$'s.
	\end{remark}
	
	\begin{remark}
		We compute the complexity of the cofinitary representation of $\bigoplus_{\c} \ZZ_2$ above as $\Sigma^0_2$, by showing that the action of the generating set above is a closed subset of ${^\omega \omega}$.
		We call this generating set $\Gamma$ and may define it by:
		\[
			g \in \Gamma \iff \forall n \in \omega : g \restr I_n \text{ acts like some } h_s \text{ on } I_n.
		\]
		Since everything to the right of the universal quantifier only happens on a finite set, it can be expressed by bounded quantifiers, and thus does not add any complexity.
	\end{remark}
	
	\section{Isomorphism types of maximal cofinitary groups}
	
	In this section, for any finite group $F$ we will construct an arithmetic maximal cofinitary group of isomorphism type $(\bigast_\c \ZZ) \times F$.
	The idea of the construction comes from the following recent result by Millhouse and the author:
	
	\begin{theorem*}[S., Millhouse, 2025]
		If there is a $\Sigma^1_2$ generating set for a free maximal cofinitary group, which is also maximal as an eventually different family of permutations, then there also is a $\Pi^1_1$ generating set for a maximal cofinitary group of the same size, which is also maximal as an eventually different family of permutations.
	\end{theorem*}

	That proof essentially constructs a maximal cofinitary group of isomorphism type $G \times \ZZ_2$, where $G$ is a freely generated maximal cofinitary group.
	In our context here, we may generalize the construction in several ways: We show that we may take the product with any finite group instead of $\ZZ_2$, and we do not need to assume that $G$ is freely generated.
	Further, we directly work with the whole group instead of generating sets.
	However, this more general setting does not apply to the context of the original proof as there some additional coding is taking place.
	Towards the proof of the theorem, we start with the following simple observation:
	\begin{lemma}
		Let $G \actson X$ be cofinitary and $H \actson Y$ be free and $Y$ be finite.
		Then the induced action $G\times H \actson X \times Y$ is cofinitary.
	\end{lemma}
	
	\begin{proof}
		The induced group action $G\times H \actson X\times Y$ is given by
		\[
		(g,h).(x,y) := (g.x,h.y).
		\]
		Let $(g,h) \in G \times H \setminus \simpleset{(e,e)}$.
		If $h \neq e$, then $h.y \neq y$ for all $y \in Y$ as $H \actson Y$ is free.
		Thus, in this case $(g,h)$ has no fixpoints either.
		Now, assume that $h = e$.
		Then $(g,e).(x,y) = (g.x,y) = (x,y)$ exactly iff $g.x = x$. 
		Thus, $\fix((g,e)) = \fix(g) \times Y$ is finite, as $g$ is cofinitary and $Y$ is finite.
	\end{proof}
	
	Thus, it is already clear that we will define the cofinitary action of $G \times F \actson \omega \times F$ as the induced product action of $G \actson \omega$ and the left regular action $F \actson F$.
	The difficult part is proving that the product action $G \times F \actson \omega \times F$ is maximal if the action $G \actson \omega$ is maximal.
	In fact, the author does not know if this is true (see discussion in the last section); instead we assume that $G \actson \omega$ is maximal as an eventually different family of permutations.
	This is fine, as all known constructions of definable maximal cofinitary groups yield this stronger type of maximality.
	
	Towards maximality, we will need to use some infinite graph theory as a tool.
	For a reference of all notions and results used here, see the infinite graph theory chapter in \cite{Diestel_2017}.
	Remember, a perfect matching $P$ in a graph $G$ is a subset of the edges of $G$ such that every vertex of $G$ is incident to exactly one edge in $P$. $(A,B)$ is called a bipartition of $G$ iff they partition its vertex set and $G$ has only edges between $A$ and $B$.
	For finite bipartite graphs Hall's marriage theorem gives a necessary and sufficient condition for perfect matchings to exist.
	The theorem does not generalize to all countable infinite graphs, but it does to locally finite graphs:
	
	\begin{theorem*}
		Let $G$ be a locally finite graph with bipartition $(A, B)$, satisfying
		\begin{enumerate}
			\item For every finite subset $S \subseteq A$ we have $\left|\text{N}(S)\right| \geq \left|S\right|$,
			\item For every finite subset $S \subseteq B$ we have $\left|\text{N}(S)\right| \geq \left|S\right|$.
		\end{enumerate}
		Then $G$ has a perfect matching.
	\end{theorem*}

	A graph $G$ is called $k$-regular iff every vertex has degree $k$.
	It is easy to see that for $k \in \omega \setminus 2$ every bipartite $k$-regular graph satisfies the assumptions of Halls's marriage theorem.
	Thus:

	\begin{corollary*}
		Let $k \in \omega \setminus 2$ and $G$ be $k$-regular.
		Then $G$ has a perfect matching.
	\end{corollary*}

	We use this corollary to prove the following crucial lemma:

	\begin{lemma}
		Let $g:\omega \times k \to \omega$ be $k$-to-$1$, i.e.\ every $m$ has exactly $k$ preimages under $g$.
		Then there is a function $i:\omega \to k$ such that $h(n) := g(n, i(n))$ is a permutation of $\omega$.
	\end{lemma}	

	\begin{proof}
		We define a graph with bipartition $(\set{L_n}{n \in \omega}, \set{R_n}{n \in \omega})$ with possible multi-edges.
		For $n \in \omega$ and $i \in k$ we have an edge $e_{n,i}$ from $L_n$ to $R_{g(n, i)}$.
		Then, every $L_n$ has degree $k$ and since $g$ is $k$-to-$1$, also every $R_n$ has degree $k$, so the graph is $k$-regular.
		Hence, we may pick a perfect matching $P$.
		Define a function $i:\omega \to k$ by
		\[
			i(n) = l \quad \iff \quad e_{n,l} \in P.
		\]
		This function is well-defined since $P$ is a matching and every $e_{n,l}$ is incident to $L_n$.
		Similarly, $i$ is defined everywhere, as $P$ is perfect.
		It remains to show that $h(n) := g(n, i(n))$ is a permutation.
		So, let $n,m \in \omega$ and assume $g(n, i(n)) = h(n) = h(m) = g(m, i(m))$.
		By definition, both $e_{n, i(n)}$ and $e_{m, i(m)}$ are in $P$, but they are both incident to $R_{g(n, i(n))} = R_{g(m,i(m))}$.
		As $P$ is a matching, they have to be the same edge, so $n = m$.
		Thus, $h$ is injective.
		
		Now, let $m \in \omega$.
		Since $P$ is perfect, there is an edge $e_{n,l}$ incident to $R_m$.
		By definition, this implies that $h(n) = g(n,i(n)) = m$.
		Thus, $h$ is surjective.
	\end{proof}

	Finally, we can put everything together to obtain our desired result:

	\begin{theorem}
		Let $G \actson \omega$ be cofinitary and maximal as an eventually different family of permutations and $F$ be a finite group.
		Then $G \times F \actson \omega \times F$ is cofinitary and maximal as an eventually different family of permutations.
	\end{theorem}
	
	\begin{proof}
		By the previous discussion it suffices to check that the induced action is maximal as an eventually different family of permutations, so let $h \in S_{\omega \times F}$.
		For $i \in 2$ let $p_i$ be the projection to the $i$-th component.
		Then, $p_0 \circ h$ is $\left|F\right|$-to-$1$, so by the previous Lemma we may choose $f:\omega \to F$ such that $\tilde{h}(n) := p_0(h(n, f_n))$ is a permutation.
		By maximality of $G \actson \omega$ as an eventually different family of permutations, there is a $g \in G$ and $A \in \infsubset{\omega}$ with $g(n) = \tilde{h}(n)$ for all $n \in A$.
		Now, we define a function $\Psi:A \to F$ by
		\[
			\Psi(n) := p_1(h(n,f_n)) \circ f_n^{-1}.
		\]
		Since $A$ is infinite and $F$ is finite, by the pigeonhole principle choose $B \in \infsubset{A}$ and $\tilde{f} \in F$, such that $\Psi(n) = p_1(h(n,f_n)) \circ f_n^{-1} = \tilde{f}$ for all $n \in B$.
		Then, for every $n \in B$ we have
		\begin{align*}
			h(n,f_n) &= (p_0(h(n,f_n)), p_1(h(n,f_n))) = (\tilde{h}(n), p_1(h(n,f_n)) \circ f_n^{-1} \circ f_n)\\
				&= (g(n), \tilde{f} \circ f_n) = (g \times \tilde{f})(n, f_n).
		\end{align*}
		But this shows that $h =^\infty g \times \tilde{f}$, i.e.\ $G \times F \actson \omega \times F$ is maximal as an eventually different family of permutations.
	\end{proof}
	
	By the discussion in the introduction we may combine this result with the results by Mejak and in Schrittesser in \cite{MejakSchrittesser_2022} to obtain maximal cofinitary groups of the following complexity:
	
	\begin{corollary}
		Let $F$ be a finite group.
		Then, there is $\Pi^0_1$-generating set for a maximal cofinitary group isomorphic to $(\bigast_\c \ZZ) \times F$.
		Thus, the group itself has complexity $\Sigma^0_2$.
	\end{corollary}
	
	\section{Questions}

	By the theorem from the first section there are always groups of size $\c$ which have a cofinitary representation, but no maximal cofinitary representation.
	However, it is not known, whether possibly every subgroup of $S_\omega$ (consistently) has a cofinitary representation:

	\begin{question}
		Does $S_\omega$ (consistently) have a (maximal) cofinitary representation?
	\end{question}

	Remember that every countable group has a cofinitary representation, so under {\sf CH} every group of size $<\!\c$ has a cofinitary representation.
	However, we cannot have this situation with large continuum by the subsequent theorem.
	Hence, the above question would be the maximal amount of cofinitary representations one could hope for.

	\begin{theorem*}[De Bruijn, \cite{DeBruijn_1957}]
		There is a group of size $\aleph_1$ which cannot be embedded into $S_\omega$.
	\end{theorem*}

	Finally, the author would like to know if the second theorem really needed to assume that $G$ is maximal as an eventually different family of permutations:
	
	\begin{question}
		Let $G \actson \omega$ be maximal cofinitary and $F$ be a finite group.
		Is then the induced product action $G \times F \actson \omega \times F$ maximal cofinitary?
	\end{question}

	We end with a short discussion of what prohibits the canonical proof using the strategy presented above, which may be helpful for anyone trying to answer this question.
	Again, given $h \in S_{\omega \times F}$, we may find $f: \omega \to F$ such that $\tilde{h}(n) := p_0(h(n,f_n))$ is a permutation.
	Now, by maximality of $G$ we only obtain a word $\tilde{w}$ in the alphabet $G \cup \simpleset{x, x^{-1}}$, so that $\tilde{w}[\tilde{h}] =^\infty e$, where $\tilde{w}[\tilde{h}]$ means substituting $x$ for $\tilde{h}$.
	Now, it is again possible (but more work) to replace each letter in $g$ in $\tilde{w}$ from $G$ by some $(g, f_g)$, to obtain a word $w$ satisfying $w[h] =^\infty e$.
	The problem is that this only shows that $h$ cannot be added to $G \times F$ if $w[h] \neq e$.
	Otherwise, this just shows that $h$ has some non-trivial relation.
	But even if $\tilde{w}[\tilde{h}] \neq e$, this does not imply that $w[h] \neq e$ as the following example shows.
	
	Let $F = \ZZ_2$ and consider $h \in S_{\omega \times \ZZ_2}$ given by
	
	\adjustbox{scale=1,center}{
	\begin{tikzcd}
		0 \arrow[r, leftrightarrow] & 1 & 2 & 3 \arrow[r, leftrightarrow] & 4 & 5 \arrow[r, leftrightarrow] & 6 & \dots\\
		0 \arrow[rru, leftrightarrow] & 1 \arrow[r, leftrightarrow] & 2 & 3 \arrow[r, leftrightarrow] & 4 & 5 \arrow[r, leftrightarrow] & 6 & \dots
	\end{tikzcd}}
	Clearly, $h$ is cofinitary and satisfies $h^2 = e$.
	Now, a possible $\tilde{h} \in S_\omega$ is given by
	
	\adjustbox{scale=1,center}{
	\begin{tikzcd}
			0 \arrow[r] & 1 \arrow[r] & 2 \arrow[ll, bend left] & 3 \arrow[r, leftrightarrow] & 4 & 5 \arrow[r, leftrightarrow] & 6 & \dots
	\end{tikzcd}}
	(One may check that the only other possible choice for $\tilde{h}$ runs into the same issue).
	Now, the word $\tilde{w}$ given to us showing that $\tilde{h}$ cannot be added to $G$ may be $\tilde{w} = x^2$, because $\tilde{w}[\tilde{h}] = \tilde{h}^2$
	
	\adjustbox{scale=1,center}{
	\begin{tikzcd}
			0 \arrow[rr, bend right] & 1 \arrow[l] & 2 \arrow[l] & 3  & 4 & 5 & 6 & \dots
	\end{tikzcd}}
	is not cofinitary.
	But $w[h] = h^2 = e$, so this $w$ does not show that $h$ cannot be added to $G \times \ZZ_2$.

	\bibliographystyle{plain}
	\bibliography{refs}

\begin{thebibliography}{10}

\bibitem{Adeleke_1988}
S.~A. Adeleke.
\newblock Embeddings of infinite permutation groups in sharp, highly
  transitive, and homogeneous groups.
\newblock {\em Proceedings of the Edinburgh Mathematical Society},
  31(2):169–178, 1988.

\bibitem{DeBruijn_1957}
N.G. de~Bruijn.
\newblock {Embedding theorems for infinite groups}.
\newblock {\em Indag. Math.}, 19:560–569, 1957.

\bibitem{Diestel_2017}
Reinhard Diestel.
\newblock {\em Graph Theory}.
\newblock Graduate Texts in Mathematics. Springer Berlin, Heidelberg, 6th ed.
  edition, 2025.

\bibitem{FischerTornquist_2015}
Vera Fischer and Asger Törnquist.
\newblock Template iterations and maximal cofinitary groups.
\newblock {\em Fundamenta Mathematicae}, 230(3):205--236, 2015.

\bibitem{HorowitzShelah_2016}
Haim Horowitz and Saharon Shelah.
\newblock {A Borel maximal cofinitary group}.
\newblock {\em The Journal of Symbolic Logic}, pages 1--14, 2023.

\bibitem{Kastermans_2006}
Bart Kastermans.
\newblock {\em Cofinitary Groups and Other Almost Disjoint Families of Reals}.
\newblock {PhD} thesis, University of Michigan, 2006.

\bibitem{Kastermans_2009}
Bart Kastermans.
\newblock Isomorphism types of maximal cofinitary groups.
\newblock {\em The Bulletin of Symbolic Logic}, 15(3):300--319, 2009.

\bibitem{MejakSchrittesser_2022}
Severin Mejak and David Schrittesser.
\newblock Definability of maximal cofinitary groups.
\newblock 2022.

\bibitem{Truss_1985}
J.~K. Truss.
\newblock Embeddings of infinite permutation groups.
\newblock {\em Proceedings of Groups - St Andrews}, pages 335--351, 1985.

\bibitem{Zhang_2000}
Yi~Zhang.
\newblock Maximal cofinitary groups.
\newblock {\em Archive for Mathematical Logic}, 39(1):41--52, 2000.

\end{thebibliography}
	
\end{document}